\documentclass[11pt]{article}
\usepackage{amsmath, amssymb, amsthm, verbatim,enumerate,bbm,color} %DIF >
\usepackage{mathrsfs}

\usepackage[backref,colorlinks,citecolor=blue,bookmarks=true]{hyperref}

\title{Two Erd\H{o}s--Hajnal-type Theorems in Hypergraphs}

\author{Michal Amir \thanks{School of Mathematical Sciences, Tel Aviv University, Tel Aviv, 6997801, Israel. Email: michalamir@mail.tau.ac.il. Supported in part by ISF Grant 1028/16.}  \and Asaf Shapira \thanks{
School of Mathematics, Tel Aviv University, Tel Aviv 69978, Israel.
Email: asafico$@$tau.ac.il. Supported in part by ISF Grant 1028/16 and ERC Starting Grant 633509.} \and
Mykhaylo Tyomkyn
	\thanks{School of Mathematical Sciences, Tel Aviv University, Tel Aviv, 6997801, Israel. Email: tyomkynm@post.tau.ac.il.
Supported in part by ERC Starting Grant 633509.}}

\date{\today}
\allowdisplaybreaks

\theoremstyle{plain}
\newtheorem{theorem}{Theorem}[section]
\newtheorem{lemma}[theorem]{Lemma}
\newtheorem{claim}[theorem]{Claim}
\newtheorem{proposition}[theorem]{Proposition}
\newtheorem{prop}[theorem]{Proposition}

\newtheorem{corollary}[theorem]{Corollary}
\newtheorem{cor}[theorem]{Corollary}

\newtheorem{remark}[theorem]{Remark}
\newtheorem{definition}[theorem]{Definition}

\makeatletter
\def\moverlay{\mathpalette\mov@rlay}
\def\mov@rlay#1#2{\leavevmode\vtop{%
   \baselineskip\z@skip \lineskiplimit-\maxdimen
   \ialign{\hfil$\m@th#1##$\hfil\cr#2\crcr}}}
\newcommand{\charfusion}[3][\mathord]{
    #1{\ifx#1\mathop\vphantom{#2}\fi
        \mathpalette\mov@rlay{#2\cr#3}
      }
    \ifx#1\mathop\expandafter\displaylimits\fi}
\makeatother

\newcommand{\cupdot}{\charfusion[\mathbin]{\cup}{\cdot}}

\makeatletter
\renewenvironment{proof}[1][\proofname]
{\par\pushQED{\qed}
	\normalfont\topsep6\p@\@plus6\p@\relax\trivlist
	\item[\hskip\labelsep\bfseries#1\@addpunct{.}]
	\ignorespaces}
{\popQED\endtrivlist\@endpefalse}
\makeatother

\newcommand{\E}{\mathbb E}
\newcommand{\F}{\mathcal F}

\newcommand{\G}{\mathcal G}
\newcommand{\N}{\mathcal N}
\newcommand{\Prob}{\mathbb P}
\newcommand{\R}{\mathcal R}
\RequirePackage{color}\definecolor{RED}{rgb}{1,0,0}\definecolor{BLUE}{rgb}{0,0,1} %DIF PREAMBLE
\begin{document}
\date{}
\maketitle

\begin{abstract}
The Erd\H{o}s--Hajnal Theorem asserts that non-universal graphs, that is, graphs that do not contain an induced
copy of some fixed graph $H$, have homogeneous sets of size significantly larger than one can generally expect to find in
a graph. We obtain two results of this flavor in the setting of $r$-uniform hypergraphs.

A theorem of R\"odl asserts that if an $n$-vertex graph is non-universal then it contains an
almost homogeneous set (i.e one with edge density either very close to $0$ or $1$) of size $\Omega(n)$. We prove that if a
$3$-uniform hypergraph is non-universal then it contains an
almost homogeneous set of size $\Omega(\log n)$. An example of R\"odl from 1986 shows that this bound is tight.

Let $R_r(t)$ denote the size of the largest non-universal $r$-graph $\G$ so that neither $\G$ nor its complement contain
a complete $r$-partite subgraph with parts of size $t$.
We prove an Erd\H{o}s--Hajnal-type stepping-up lemma, showing how to transform a lower bound for $R_{r}(t)$ into a lower bound for $R_{r+1}(t)$.
As an application of this lemma, we improve a bound of Conlon--Fox--Sudakov by showing that $R_3(t) \geq t^{\Omega(t)}$.
\end{abstract}

\section{Introduction}\label{sec:intro}

Let us say that a set of vertices in a graph (or hypergraph) is {\em homogeneous} if it spans either a clique (i.e.~a complete graph) or an
independent set (i.e.~an empty graph). Ramsey's theorem states that every graph contains a homogeneous set of size $\frac12\log_2 n$, and
Erd\H{o}s proved that in general, one cannot expect to find a homogeneous set of size
larger than $2 \log_2 n$ (see \cite{GRS}). Since Erd\H{o}s's example uses random graphs, and random graphs are universal (with high probability), that is,
they contain an induced copy of every fixed graph $H$, it is natural to ask what happens if we assume that $G$ is non-universal, or equivalently, that it is induced $H$-free for some fixed $H$. A theorem of Erd\H{o}s and Hajnal
\cite{EH} states that in this case we are guaranteed to have a homogeneous
set of size $2^{\Omega(\sqrt{\log n})}$, that is, a significantly larger set than in the worst case. The notorious
Erd\H{o}s--Hajnal Conjecture states that one should be able to go even further and improve this bound to $n^{c}$,
where $c=c(H)$. We refer the reader to \cite{Chud} for more background on this conjecture and related results.

Conlon, Fox and Sudakov \cite{CFS} and R\"odl and Schacht \cite{RS} have recently initiated
the study of problems of this type in the setting of $r$-uniform hypergraphs (or $r$-graphs for short).
Our aim in this paper is to obtain two results of this flavor which are described in the next two subsections.

\subsection{Almost homogeneous sets in non-universal hypergraphs}

Our first result is motivated by a theorem of R\"odl \cite{RodlUniv}. Let us say that a set of vertices $W$ in
a graph is $\eta$-homogeneous if $W$ either contains at least $(1-\eta)\binom{|W|}{2}$ or at most $\eta\binom{|W|}{2}$
edges. It is a standard observation that Erd\H{o}s's lower bound for Ramsey's theorem (mentioned above), actually shows that some (actually, most) graphs of order $n$
do not even contain $\frac14$-homogeneous\footnote{One can easily replace the $\frac14$ with any constant smaller than $\frac12$. We will stick with the $\frac14$ in order to streamline the presentation.} sets of size $O(\log n)$. In other words, in the worst case relaxing $0$-homogeneity
to $\frac14$-homogeneity does not make the problem easier. R\"odl's \cite{RodlUniv} surprising theorem then states
that if $G$ is non-universal then for any $\eta > 0$, it contains an $\eta$-homogeneous set of size $\Omega(n)$, where
the hidden constant depends on $\eta$.

It is natural to ask if a similar\footnote{We of course say that a set of vertices $W$ in
an $r$-graph is $\eta$-homogeneous if $W$ either contains at least $(1-\eta)\binom{|W|}{r}$ or at most $\eta\binom{|W|}{r}$
edges} result holds also in hypergraphs. Random $3$-graphs
show that, in the worst case, the largest $\frac14$-homogeneous set in a $3$-graph might be of size $O(\sqrt{\log n})$.
Our first theorem in this paper shows that, as in graphs, if we assume that a $3$-graph is non-universal
then we can find a much larger almost homogeneous set.

\begin{theorem}\label{thm:main}
For every $3$-graph $\F$ and $\eta >0$ there is $c=c(\F,\eta)>0$ such that every induced $\F$-free $3$-graph
on $n$ vertices contains an $\eta$-homogeneous set of size $c\log n$.
\end{theorem}

R\"odl \cite{RodlUniv} found an example
of a non-universal $3$-graph in which the largest $\frac14$-homogeneous set has size $O(\log n)$.
Hence, the bound in Theorem \ref{thm:main} is tight up to the constant $c$.
We will describe in Section \ref{sec:concluding} (see Proposition \ref{prop:example}) a generalization of R\"odl's example, giving for every $r \geq 3$ an
example of a non-universal $r$-graph in which the size of the largest $\frac14$-homogeneous set is
$O((\log n)^{1/(r-2)})$. It seems reasonable to conjecture that this upper bound is tight, that is, that for every $r \geq 3$ every non-universal $r$-graph has an almost homogeneous set of size $\Omega((\log n)^{1/(r-2)})$.

Let $K_k$ denote the complete graph on $k$ vertices and let $K^{(3)}_k$ denote the complete $3$-graph on $k$ vertices.
It is easy to see that up to a change of constants, a set of vertices has edge density close to $0/1$ (i.e is $\eta$-homogeneous for some small $\eta$),
if and only if it has $K_k$-density close to $1$ either in the graph or in its complement. The same applies to $3$-graphs.
An interesting feature of the proof of Theorem \ref{thm:main} is that instead of gradually building a set of vertices with very large/small edge density, we find it easier to build such a set with large $K^{(3)}_k$-density either in $\G$ or its complement.
The way we gradually build such a set is by applying a variant of a greedy embedding scheme used by R\"odl and Schacht \cite{RS} in order
to give an alternative proof of an elegant theorem of Nikiforov \cite{Nikiforov} (this alternative proof is also implicit in \cite{CFS}). To get this embedding scheme `started' we prove
a lemma saying that if a $3$-graph $\G$ is non-universal then there is a graph $G$ on a subset of $V(\G)$ such that either almost all or almost none of the $K_k$'s of $G$ are also $K^{(3)}_k$'s in $\G$. This latter statement is proved via the hypergraph regularity method.

\subsection{Complete partite sets in non-universal hypergraphs}

Determining the size of the largest homogeneous set in a $3$-graph is still a major open problem, see \cite{CFShyp}.
The best known lower and upper bounds are of order $\log\log n$ and $\sqrt{\log n}$ respectively. It is thus hard to formulate
a $3$-graph analogue of the Erd\H{o}s--Hajnal Theorem since it is not clear which bound one is trying to beat.
At any rate, as of now, we do not even know if a non-universal $3$-graph contains
a homogeneous set of size $\omega(\log\log n)$ (see Section \ref{sec:concluding} for further discussion on this problem).
This motivated the authors of \cite{CFS,RS} to look at the following related problem.
Let $K^{(3)}_{t,t,t}$ denote the complete $3$-partite $3$-graph with each part of size $t$.
It is a well known fact \cite{Erdos} that every $3$-graph of positive density contains a copy of $K^{(3)}_{t,t,t}$
with $t=\Omega(\sqrt{\log n})$. This immediately means that for every $3$-graph $\G$, either $\G$ or its complement contains a
$K^{(3)}_{t,t,t}$ with $t=\Omega(\sqrt{\log n})$. As evidenced by random 3-graphs, this bound is tight.
A natural question, which was first addressed by Conlon, Fox and Sudakov \cite{CFS} and by R\"{o}dl and Schacht \cite{RS} is
whether one can improve upon this bound when $\G$ is assumed to be non-universal.

It will be more convenient to switch gears at this point, and let $R_{3,\F}(t)$ denote
the size of the largest induced $\F$-free $3$-graph $\G$, so that neither $\G$ nor $\overline{\G}$ contain
a copy of $K^{(3)}_{t,t,t}$. So the question posed at the end of the previous paragraph is equivalent
to asking if for every fixed $\F$ we have $R_{3,\F}(t) \leq 2^{o(t^2)}$, and the results of \cite{CFS,RS}
establish that this is indeed the case \footnote{While the proof in \cite{RS} obtained the bound $R_{3,\F}(t)\leq 2^{t^2/f(t)}$ with $f(t)$ an inverse Ackermann-type function (on account of using the hypergraph regularity lemma), the proof in \cite{CFS} gave the improved bound $R_{3,\F}(t)\leq 2^{t^{2-c}}$
where $c=c(\F)$ is a constant that depends only on $\F$ (but goes to zero with $|V(\F)|$).}.
Conlon, Fox and Sudakov \cite{CFS} also found an example of a $3$-graph $\F$ for which
$R_{3,\F}(t) \geq 2^{\Omega(t)}$. Our second result improves their lower bound as follows.

\begin{theorem}\label{thm:example}
There is a $3$-graph $\F$ for which $R_{3,\F}(t) \geq t^{\Omega(t)}$.
\end{theorem}

As discussed in \cite{CFS}, it is natural to consider the corresponding problem in general $r$-graphs.
Letting $K^{(r)}_{t,\dots,t}$ denote the complete $r$-partite $r$-graph with parts of size $t$,
we define $R_{r,\F}(t)$ to be the size of the largest induced $\F$-free $r$-graph $\G$, so that neither $\G$ nor $\overline{\G}$ contain
a copy of $K^{(r)}_{t,\dots,t}$. It follows from \cite{Erdos}, which establishes that in {\em every} $r$-graph $\G$ on $2^{\Omega(t^{r-1})}$
vertices with density $1/2$ we can find a $K^{(r)}_{t,\dots,t}$, that
\begin{equation}\label{easybound1}
R_{r,\F}(t) \leq 2^{O(t^{r-1})}\;.
\end{equation}
It was shown in \cite{CFS} that there is an $r$-graph $\F$ satisfying
\begin{equation}\label{easybound2}
R_{r,\F}(t) \geq 2^{\Omega(t^{r-2})} \;.
\end{equation}
An alternative proof of (\ref{easybound2}) follows from Proposition \ref{prop:example}.

The famous step-up lemma of Erd\H{o}s and Hajnal (see \cite{GRS}) allows one to transform a construction of an $r$-graph without a large monochromatic set into an exponentially larger $(r+1)$-graph without a large monochromatic set (of roughly the same size).
Observe that both (\ref{easybound1}) and (\ref{easybound2}) suggest that if $2^{t^{\alpha}}$ is the size of the largest non-universal
$r$-graph $\G$, so that neither $\G$ nor $\overline{\G}$ contain $K^{(r)}_{t,\dots,t}$, then the corresponding bound for $(r+1)$-graphs is $2^{t^{\alpha+1}}$. The following theorem establishes one side of this relation, by proving an Erd\H{o}s--Hajnal-type step-up lemma for
the problem of bounding $R_{r,\F}(t)$.

\begin{theorem}\label{thm:stepup}
The following holds for every $r \geq 4$. For every $(r-1)$-graph $\F$
there is an $r$-graph $\F^+$ and a constant $c=c(r,\F)>0$, so that
$$
R_{r,\F^+}(t) \geq \left(R_{r-1,\F}(ct)\right)^{ct}\;.
$$
\end{theorem}

Theorem \ref{thm:stepup} implies that any improvement of (\ref{easybound2}) for $r=3$ immediately implies a similar improvement of (\ref{easybound2}) for arbitrary $r \geq 3$. In particular, as a corollary of Theorems \ref{thm:example} and \ref{thm:stepup} we obtain the following improvement of (\ref{easybound2}).

\begin{corollary}\label{cor:lower}
For every $r \geq 3$ there is an $r$-graph $\F$ satisfying $R_{r,\F}(t) \geq t^{\Omega(t^{r-2})}$.
\end{corollary}

To prove Theorem \ref{thm:stepup} we need to overcome two hurdles. First, we need a way to construct the $r$-graph $\F^+$ given the
$(r-1)$-graph $\F$. An important tool for this step will be an application of a
theorem of Alon, Pach and Solymosi \cite{APS}, which is a hypergraph extension of a result of R\"odl and Winkler \cite{RW}.
The second hurdle is how to construct an $r$-graph avoiding a $K^{(r)}_{t,\dots,t}$ given an $(r-1)$-graph avoiding
a large $K^{(r-1)}_{t',\dots,t'}$. Here we will apply a version of a very elegant argument from \cite{CFShyp}, which is a variant of the Erd\H{o}s--Hajnal step-up lemma. While this variant of the step-up lemma is not as efficient as the original one\footnote{Observe that step-up lemmas with an exponential blowup-up are {\em not} useful in our setting since (\ref{easybound1}) and (\ref{easybound2}) tell us that the gap between $R_{r-1,\F}(t)$ and $R_{r,\F^+}(t)$ is {\em not} exponential.}, it is strong enough for our purposes.

It would be very interesting to further narrow the gap between (\ref{easybound1}) and Corollary \ref{cor:lower}.
The most natural question is if one can extend the results of \cite{CFS,RS} by showing
that $R_{r,\F}(t) \leq 2^{o(t^{r-1})}$ for every $r \geq 4$ (the case $r=3$ was handled in \cite{CFS,RS}).
We should mention that \cite{CFS} conjectured that in fact $R_{r,\F}(t) \leq 2^{t^{r-2+o(1)}}$. Observe that by Theorem \ref{thm:stepup}, showing
that $R_{3,\F}(t) \geq 2^{t^{1+c}}$ for some $3$-graph $\F$ and $c>0$ would disprove this conjecture for all $r \geq 3$.

\subsection{Paper overview}
The rest of the paper is organized as follows.
In Section \ref{sec:proof} we give the proof of Theorem \ref{thm:main}, deferring the proof of one
of the key lemmas to Section \ref{sec:regularity}. The proof of Theorems \ref{thm:example} and \ref{thm:stepup} is given
in Section \ref{sec:examples}. Section \ref{sec:concluding} contains some concluding remarks including the statement and proof of
Proposition \ref{prop:example} which gives a generalization of R\"odl's example, establishing that the bound in Theorem \ref{thm:main}
is asymptotically tight.

We are following the common practice of ignoring rounding issues for a better proof transparency. Throughout the paper, $\log$ denotes the natural logarithm, $\mathbb{N}$ stands the set of positive integers, and for $n\in \mathbb{N}$ we write $[n]$ for the set of integers $\{1,\dots,n\}$.

\section{Proof of Theorem \ref{thm:main}}\label{sec:proof}

In this section we give the proof of Theorem \ref{thm:main}, save for one lemma that is proved in Section \ref{sec:regularity}.
The proof of Theorem \ref{thm:main} will rely on Lemmas \ref{lem:main} and \ref{lem:kcase} stated below. We start with a few definitions.
All \emph{graphs} (or $2$-graphs) will be simple and undirected, and will be denoted by capital letters e.g. $G,H$, while \emph{$r$-graphs} will be denoted by script letters e.g. $\G, \mathcal{H}$.
For a graph $G$ and a $3$-graph $\G$, defined on the same vertex set, we say that $G$ \emph{underlies} $\G$ if every edge of $\G$ is a triangle in $G$. We use $\N_k(G)$ and $\N_k(\G)$ to denote the number of cliques $K_k$ and hypercliques  $K_k^{(3)}$ in a graph $G$ and a $3$-graph $\G$, respectively. For an integer $k\geq 3$, we say that a graph (resp. $3$-graph) is \emph{$k$-partite} on disjoint vertex classes $V_1,\ldots,V_k$ if each edge of the graph (resp. $3$-graph) has at most one vertex in each of the $k$ vertex sets.
Observe that if a $k$-partite $G$ on vertex sets $V_1,\ldots,V_k$ underlies $\G$, then $\G$ is also $k$-partite with respect to these vertex sets.
%(or that $\G$ is underlied by $G$)

\begin{definition}[$(\epsilon,\eta)$-dense graph]\label{def:under}
Suppose $G$ is a $k$-partite graph on disjoint vertex sets $V_1,\ldots,V_k$.
We say that $G$ is $(\epsilon,\eta)$-dense with respect to some $3$-graph $\G$ if
\begin{enumerate}
\item $G$ underlies $\G$.
\item $\N_k(G) \geq \epsilon \prod_{i}|V_i|$. %of the $k$-tuples of vertices from $V_1 \times \ldots \times V_k$ form a $K_k$ in $G$.
\item $\N_k(\G) \geq (1-\eta)\N_k(G)$.
\end{enumerate}
\end{definition}

We are now ready to state the two key lemmas needed to prove Theorem \ref{thm:main}, and then show how to derive it from them.
In what follows we use $G[W]$ (or $\G[W]$) to denote the graph (or $3$-graph) induced by a set of vertices $W$.
As usual, we use $\overline{\G}$ to denote the complement of a $3$-graph $\G$.

\begin{lemma}\label{lem:main}
For every $3$-graph $\F$ and for any $k\in \mathbb{N}$ and $0<\eta<1$ there exist $c'=c'(\F,k,\eta)>0$ and $\epsilon=\epsilon(\F,k,\eta)>0$ as follows.
If $\G$ is an induced $\F$-free $3$-graph on $n$ vertices, then there are $k$ disjoint vertex sets $V_1,\dots, V_k\subseteq V(\G)$ of equal size at least $c' n$ and a $k$-partite graph $G$ on $V_1,\ldots,V_k$ which is $(\epsilon,\eta)$-dense either with respect to a subgraph of $\G$ or with respect
to a subgraph of $\overline{\G}$.
\end{lemma}

\begin{lemma}\label{lem:kcase}
For every $\epsilon>0$ and integer $k\geq 3$ there exists $c''=c''(\epsilon,k)>0$ as follows.
Suppose $G$ is a $k$-partite graph on vertex sets $V_1,\ldots,V_k$ of size $m$ each,
and $G$ is $(\epsilon,\eta)$-dense with respect to some $3$-graph $\G$.
Then there are subsets $W_{j}\subseteq V_{j}$ for $j=1,\ldots,k$ of equal size at least $c''\log m$ such that setting $W=\bigcup_i W_i$ we have
\begin{equation}\label{eq3}
\N_k(\G[W])\geq (1-2^{\binom{k}{2}}\eta)\prod_{\ell=1}^k|W_\ell|\;.
\end{equation}
%at least $(1-2^{\binom{k}{2}}\delta)\prod_{\ell=1}^k|W_\ell|$ of the $k$-tuples
%in $W_1\times \cdots \times W_k$ form a $K^{(3)}_k$ in $\G$.
\end{lemma}

%Using the above two lemmas we will prove the following lemma (alluded to in Section \ref{sec:intro}).

\begin{proof}[Proof of Theorem~\ref{thm:main}]
Given $\F$ and $\eta$, set $k=10/\eta$ and let $c'=c'(\F,k,\eta/2^{k^2})$ and $\epsilon=\epsilon(\F,k,\eta/2^{k^2})$ be the constants from Lemma \ref{lem:main}.
If $G$ is induced $\F$-free then by Lemma \ref{lem:main} we can find $k$ disjoint vertex sets $V_1,\ldots,V_k$ of equal size at least $c'n$
and a $k$-partite graph on these sets which is $(\epsilon,\eta/2^{k^2})$-dense with respect to either a subgraph of $\G$ or a subgraph $\overline{\G}$.
Suppose, without loss of generality, that the former case holds. By Lemma \ref{lem:kcase} we can find subsets
$W_1 \subseteq V_1,\ldots,W_k \subseteq V_k$, each of size $c''\log |V_i| \geq c''\log (c'n) \geq c \log n$
satisfying\footnote{Note that this inequality would in fact hold for a subgraph of $\G$, i.e. the one we got from Lemma \ref{lem:main}. By monotonicity
we can work with $\G$ itself.}
$\N_k(\G[W])\geq (1-\eta/2)\prod_{\ell=1}^k|W_\ell|$.

Setting $W=\bigcup_i W_i$ we now claim that $e(\G[W]) \geq (1-\eta)\binom{|W|}{3}$.
To see this we first observe that since $k= 10/\eta $ and all sets $W_i$ are of the same size,
then at most $\frac12\eta{\binom{|W|}{3}}$ of the $3$-tuples in
$W$ do {\em not} belong to $3$ distinct sets $W_i$. We also observe that for each triple of distinct
sets $W_p,W_q,W_r$, the $3$-graph $\G$ has at least $(1-\eta/2)|W_p||W_q||W_r|$ edges with one vertex in each of the sets.
Indeed, if this is not the case then $\G$ cannot contain the number of copies of $K^{(3)}_k$ asserted at the end of the previous paragraph.
It is finally easy to see that the above two observations imply that $e(\G[W]) \geq (1-\eta)\binom{|W|}{3}$.
\end{proof}

The rest of this section is devoted to the proof of Lemma \ref{lem:kcase}.
The proof of Lemma \ref{lem:main} appears in the next section.
The proof of Lemma~\ref{lem:kcase} is carried out via repeated application of Lemma~\ref{lem:eqsize} stated below.
For its proof, we will need the following special case of the classical K\H{o}v\'{a}ri-S\'os-Tur\'an Theorem~\cite{KST} (similar lemmas where used in \cite{CFS,Nikiforov,RS}).. For completeness we include the proof.

\begin{lemma}\label{lem:KST2}
Suppose that $\epsilon>0$ and $G$ is a bipartite graph between sets $A$ and $B$ with $e(G)\geq\epsilon |A||B|$.
Suppose further that $|A| \leq \frac12\log |B|$ and $s,t>0$ satisfy $s=\epsilon|A|$ and $t=\sqrt{|B|}$.
Then $G$ contains a copy of $K_{s,t}$ with $s$ vertices in $A$ and $t$ vertices in $B$.
\end{lemma}
\noindent
\begin{proof}
Write $d(v)$ for the degree of $v$ in $G$. Then the number of copies of $K_{s,1}$, with $s$ vertices in $A$ and one vertex in $B$ is just $\sum_{v\in B}\binom{d(v)}{s}$.
By convexity of $f(z)=\binom{z}{s}$ we have
$$
\sum_{v\in B}\binom{d(v)}{s}\geq |B|\binom{\frac{1}{|B|}\sum_{v\in B}d(v)}{s}= |B|\binom{e(G)/|B|}{\epsilon |A|}\geq |B|\;.
$$
Since $A$ has fewer than $2^{|A|}$ subsets of size $s$, by the pigeonhole principle some set $U\subset A$ of size $s$

appears in at least $2^{-|A|}|B|\geq \sqrt{|B|}=t$ of these copies of $K_{s,1}$. It is clear that $U$ and the $t$ vertices from $B$ participating
in these $t$ copies of $K_{s,1}$ form a $K_{s,t}$.
\end{proof}

If $U,V$ are two disjoint vertex sets in a graph $G$, we use $G[U,V]$ to denote the bipartite subgraph of $G$ induced
by these sets, that is, the graph containing all edges with one vertex in each set. More generally, for disjoint vertex sets $V_1,\ldots,V_k$ we use $G[V_1,\ldots,V_k]$ to denote the $k$-partite subgraph containing all edges of $G$ connecting two distinct sets $V_i,V_j$.
Similarly, if $\G$ is a $3$-graph then we use $\G[V_1,\ldots,V_k]$ to denote the $k$-partite $3$-graph containing
all edges of $\G$ connecting three distinct sets.

\begin{lemma}\label{lem:eqsize}
Suppose that $\eta>0$, $0< \epsilon < 1$, and $0<\gamma \leq 1/2$ are constants, $V_1,\ldots,V_k$ are disjoint
vertex sets in some $3$-graph $\G$, and $G$ is a $k$-partite graph on $V_1,\ldots,V_k$
which is $(\epsilon,\eta)$-dense with respect to $\G$. If for some pair of indices $1\leq i<j\leq k$ we have $|V_j|=m$ and $|V_i|\geq \gamma \log m$,
then there exist subsets $S_i\subseteq V_i$ and $S_j\subseteq V_j$ such that
\begin{itemize}
\item $|S_{i}|= \epsilon \gamma\log m$ and $|S_{j}| = m^{1/2}$,
\item $G[S_{i},S_{j}]$ is complete bipartite,
\item $G[V_1,\dots,S_i,\dots,S_{j},\dots,V_k]$ is $(\epsilon/4,2\eta)$-dense with respect to $\G[V_1,\dots,S_i,\dots,S_{j},\dots,V_k]$.
\end{itemize}
\end{lemma}
\begin{proof}
Observe that by the conditions of the lemma, for any choice of $S_i,S_j$ we have that the $k$-partite graph $G[V_1,\dots,S_i,\dots,S_{j},\dots,V_k]$ underlies the $k$-partite $3$-graph $\G[V_1,\dots,S_i,\dots,S_{j},\dots,V_k]$, so we will not need to worry about item $1$ of Definition \ref{def:under}.

For each $1\leq p\neq q\leq k$ we use shorthand $G_{p,q}$ for the bipartite graph $G[V_p,V_q]$. For an edge $e\in E(G)$ let $\N_k(G,e)$ and $\N_k(\G,e)$ stand for the number of $k$-cliques in $G$ and $\G$ respectively, containing the vertices of $e$. Call an edge $e\in E(G_{i,j})$ \emph{good} if $\N_k(\G,e)\geq (1-2\eta)\N_k(G,e)$, otherwise the edge is \emph{bad}.
Since $G$ is $(\epsilon,\eta)$-dense with respect to $\G$ we have
\begin{align*}
\sum_{e\ good}^{} \N_k(G,e) + \sum_{e\ bad}^{} (1-2\eta)\N_k(G,e) \geq \N_{k}(\G) \geq (1-\eta)\N_{k}(G)\;,
\end{align*}
implying that $\sum_{e\ bad}^{} \N_k(G,e) \leq\frac12 \N_{k}(G)$. This, and the assumption that $G$ is $(\epsilon,\eta)$-dense, imply that
\begin{equation}\label{eq4}
\sum_{e\ good}^{} \N_k(G,e) \geq \frac12 \N_{k}^{}(G)\geq \frac{\epsilon}{2}\prod_{\ell=1}^k|V_\ell|{}\;.
\end{equation}
Let $F''$ be the subgraph of $G_{i,j}$ consisting of all good edges, and let $G''$ be the graph formed by $F''$ and $G_{p,q}$ for all $\{p,q\}\neq \{i,j\}$.
By (\ref{eq4}), we have $\N_k(G'') \geq \frac{\epsilon}{2}\prod_{\ell=1}^k|V_\ell|$.
Let $F\subseteq F''$ be the set of all edges $e$ such that $\N_k(G,e)\geq \frac{\epsilon}{4} \prod_{\ell\notin \{i,j\}}|V_\ell|$.
Note that we must have $e(F)\geq \frac{\epsilon}{4}|V_i||V_j|{}$, as otherwise
\begin{align*}
\N_{k}^{}(G'')<\frac{\epsilon}{4}|V_i||V_j|\cdot \prod_{\ell\notin \{i,j\}} |V_\ell| + |V_i||V_j|\cdot \frac{\epsilon}{4}\prod_{\ell\notin \{i,j\}}|V_\ell| = \frac{\epsilon}{2}\prod_{\ell=1}^k|V_\ell|\;,
\end{align*}
which would contradict the lower bound on $\N_k(G'')$ mentioned above.
%Here we use the (trivial) fact that

%$\N_k(G,e)\leq \prod_{\ell\notin \{i,j\}}|V_\ell|$ for every $e\in F''$.

Since $|V_i|\geq \gamma \log m$, by averaging there exists a subset $A\subset V_i$ with $|A|=\gamma \log m$ and
\begin{equation}\label{eq:k1}
e(F[A,V_j])\geq \frac{\epsilon}{4}|A||V_j|\;.
\end{equation}
By~\eqref{eq:k1}, we can apply Lemma~\ref{lem:KST2} to the graph $F[A,V_j]$ to obtain subsets $S_i\subseteq A\subseteq V_i$ and $S_{j}\subseteq V_{j}$ with $|S_{i}|=\epsilon\gamma\log m$ and $|S_{j}|=m^{1/2}$, such that $G[S_{i},S_{j}]$ is complete bipartite.

We now look at $G'=G[V_1,\dots,S_{i},\dots,S_{j}\dots,V_k^{}]$. Since $G[S_{i},S_{j}]$ is complete bipartite and for every edge $e\in E(G[S_{i},S_{j}])$ we have that $\N_k(G,e)\geq \frac{\epsilon}{4} \prod_{\ell \notin \{i,j\}}|V_\ell|$, it follows that
	
\begin{align*}
\N_k(G')\geq \frac{\epsilon}{4} |S_{i}||S_{j}|\prod_{\ell\notin \{i,j\}}|V_\ell|\;.
\end{align*}

Moreover, since every $e\in E\left(G[S_{i},S_{j}]\right)$ is good, we have
$$
\N_k(\G[V_1,\dots,S_{i}\dots,S_j,\dots V_k]) \geq(1-2\eta) \N_k(G')\;,
$$
implying that $G'$ is indeed $(\epsilon/4,2\eta)$-dense with respect to $\G[V_1,\dots,S_{i},\dots,S_{j}\dots,V_k]$.
\end{proof}

\begin{proof}[Proof of Lemma~\ref{lem:kcase}]
Set $\gamma = \frac12$ and apply Lemma~\ref{lem:eqsize} with $i=1$ and $j=2$ to obtain subsets $S_1^1\subseteq V_1$ and $S_2\subseteq V_2$ with $|S_1^1|=\epsilon \gamma \log m$ and $|S_2|=m^{1/2}$, and with the properties stated in Lemma~\ref{lem:eqsize}. Note now that the conditions of Lemma~\ref{lem:eqsize} are satisfied in the graph $G[S_1^1,S_2,V_3,\dots,V_k]$, with $2\eta$, $\epsilon/4$, $\epsilon \gamma$, $S_1^1$ and $V_3$  playing the role of $\eta$, $\epsilon$, $\gamma$, $V_i$ and $V_j$  respectively, so Lemma~\ref{lem:eqsize} can be applied again to obtain subsets $S_1^2\subseteq S_1$ and $S_3\subseteq V_3$ with properties as stated therein. Continuing for $i=1$ and (consecutively) $j=2,\dots,k$ we obtain a set $S_1^{k-1}=:U_1\subseteq V_1$
%(the set $S_1$ resulting from the last application of Lemma~\ref{lem:eqsize})
of size $c_1\log m$ for a constant $c_1=c_1(\epsilon,k)$, and sets $S_2,\dots,S_k$ of size $m^{1/2}$ each, so that all graphs $G[U_1,S_j]$ are complete bipartite
and $G[U_1,S_2,\dots,S_k]$ is $(\epsilon/4^{k-1}, 2^{k-1}\eta)$-dense with respect to
$\G[U_1,S_2,\dots,S_k]$.
%$$
%\N_k(G[U_1,S_2,\dots,S_k])\geq \frac{\epsilon}{4^{k-1}}|U_1|\prod_{j=2}^k|S_j|\;,
%$$
%and
%$$
%\N_k(\G[U_1,S_2,\dots,S_k])\geq (1-2^{k-1}\eta) \N_k(G[U_1,S_2,\dots,S_k])\;.
%$$

Next, since $\log (m^{1/2})=(1/2)\log m$, the above procedure can be applied again for $i=2$ and $j=3,\dots,k$ to obtain $U_2\subseteq V_2$ of size $c_2\log m$ and (abusing notation) sets $S_3\subseteq V_3,\dots, S_k\subseteq V_k$ of size $m^{1/4}$ each, with analogous properties.
Continuing in the same way for each $i$ yields sets $U_i\subseteq V_i$ for $i=1,\dots,k-1$ with $|U_i|=c_i\log m$ and $U_k=S_k\subseteq V_k$ (the output of the last application of Lemma~\ref{lem:eqsize}) with $|U_k|=m^{1/2^{k-1}} \geq c_k\log m$ such that $G^*=G[U_1,\dots,U_k]$ is a complete $k$-partite graph.
Moreover, by Lemma~\ref{lem:eqsize} (which we applied $\binom{k}{2}$ times), $G^{*}$ is $(\epsilon/4^{\binom{k}{2}},2^{\binom{k}{2}}\eta)$-dense with respect to $\G[U_1,\dots, U_k]$. In particular,
$$\N_k(\G[U_1,\dots,U_k])\geq (1-2^{\binom{k}{2}}\eta)\N_k(G^{*}).
$$
However, since (crucially) $G^*$ is complete $k$-partite, this means that
$$\N_k(\G[U_1,\dots,U_k])\geq (1-2^{\binom{k}{2}}\eta)\prod_{\ell=1}^k |U_\ell|.
$$
%$\G[U_1,\dots,U_k]$ contains
%at least $(1-2^{\binom{k}{2}}\eta)\prod_{\ell=1}^k |U_\ell|$ copies of $K^{(3)}_k$.

To complete the proof we just need to pick sets of equal size. Putting $c''=\min\{c_1,\dots,c_k\}$, by averaging,
there exist subsets $W_\ell\subseteq U_\ell$ for $\ell=1,\dots,k$ of equal size $c''\log m$, so that setting $W=\bigcup_i W_i$ we have
$$
\N_k(\G[W]) \geq \N_k(\G[W_1,\dots,W_k])\geq (1-2^{\binom{k}{2}}\eta)\prod_{\ell=1}^k |W_\ell|\;,
$$
as we had to show.
\end{proof}

%such that in the graph $G^{final}=G[W_1,\dots,W_k]$ (which is complete $k$-partite, being a subgraph of $G^*$) also the proportion of at least %$1-2^{\binom{k}{2}}\eta$ of $k$-cliques underlie a $K_k^{(3)}\subseteq \G$. In other words,

\section{Near-homogeneous multipartite hypergraphs}\label{sec:regularity}

In this Section we prove Lemma~\ref{lem:main}.
In the first subsection we will discuss some tools from the hypergraph regularity method that are needed for the proof of the lemma. We will follow the definitions from~\cite{NPRS} and~\cite{RRS}. The proof itself appears in the second subsection.

\subsection{A primer on hypergraph regularity}

We first recall the definition of a (Szemer\'{e}di)-regular bipartite graph.
\begin{definition}
Let $d_2, \delta_2>0$. A bipartite graph $G$ with the bipartition $V(G)=X\cupdot Y$ is called \emph{$(\delta_2, d_2)$-regular} if for all subsets $X'\subseteq X$ and $Y'\subseteq Y$ we have
$$
\left|e(G[X',Y'])-d_2|X'||Y'|\right|\leq \delta_2 |X||Y|\;.
$$
\end{definition}
Extending this notion to $3$-graphs, we need the following definition of a regular $3$-partite $3$-graph~\cite[Definition 3.1]{RRS}. For a graph $G$ write $T(G)$ for the set of triangles in $G$.
\begin{definition}
Let $d_3,\delta_3>0$.
%and $r\in \mathbb{N}$
A $3$-graph $\mathcal{H}=(V,E_H)$ is called \emph{$(\delta_3,d_3)$-regular}
%r
with respect to a tripartite graph $P=(X \cupdot Y \cupdot Z,E_P)$ with $V\supseteq X\cup Y\cup Z$, if for every subgraph $Q\subseteq P$ we have
\begin{equation}\label{eq:hypreg}
\left||E_H \cap T(Q)| - d_3|T(Q)|\right| \leq \delta_3|T(P)|\;.
\end{equation}
\end{definition}
\noindent
We say that $\mathcal{H}$ is \emph{$\delta_3$-regular} with respect to $P$ if it is $(\delta_3,d_3)$-regular for an unspecified $d_3$. Also, define $$d(\G|P):=\frac{|E(\G)\cap T(P)|}{|T(P)|}$$ to be the \emph{relative density} of $\G$ with respect to $T(P)$. By putting $Q=P$ in~\eqref{eq:hypreg}, we obtain:
\begin{remark}\label{rem:regdens}
If $\mathcal{H}$ is $(\delta_3,d_3)$-regular with respect to some graph $P$, then
$$d_3-\delta_3\leq d(\G|P)\leq d_3+\delta_3.
$$
\end{remark}

Next, we will state the regularity lemma for $3$-graphs. Informally speaking, it is similar in spirit to the classical Szemer\'{e}di regularity lemma for $2$-graphs, in that a large hypergraph is being partitioned into a bounded number of fragments, almost all of which are regular. One major difference between the $2$-graph and $3$-graph cases is that, due to various technical issues, for $3$-graphs we partition not just $V(\mathcal{H})$ into subsets $V_0 \cupdot V_1 \cupdot \dots \cupdot V_t$, but also the ($2$-uniform) edge sets of the complete bipartite graphs $K(V_i,V_j)$ between $V_i$ and $V_j$ into sparser bipartite graphs. This will naturally give rise to a number of tripartite $2$-graphs (`triads'); the lemma states then that $\mathcal{H}$ will be regular with respect to most of them.

To give a precise formulation, we shall be using the following version of the regularity lemma for $3$-graphs~\cite[Theorem 3.2]{RRS}.
\begin{prop}[Regularity Lemma]\label{prop:reglemma}
For all $\delta_3>0$, $\delta_2:\mathbb{N}\rightarrow (0,1]$ and $t_0\in \mathbb{N}$ there
exist $T_0, n_0 \in \mathbb{N}$ such that for every $n \geq n_0$ and every $3$-graph
$\mathcal{H}$ with $|V(\mathcal{H})|=n$ the following holds.

There are integers $t_0\leq t \leq T_0$ and $\ell \leq T_0$ such that there exists a  partition $V_0 \cupdot V_1 \cupdot \dots \cupdot V_t = V(\mathcal{H})$, and for all $1\leq i < j \leq t$ there exists a partition of $K(V_i,V_j)$
$$\mathcal{P}^{ij}=\{P^{ij}=(V_i\cupdot V_j, E_\alpha^{ij}):1\leq \alpha \leq \ell \},
$$
satisfying the following properties:
\begin{enumerate}
\item [(i)] $|V_0| \leq \delta_3n$ and $|V_1|=\dots=|V_t|$,
\item [(ii)] for all $1\leq i<j\leq t$ and $\alpha \in [\ell]$ the bipartite graph $P_{\alpha}^{ij}$ is $(\delta_2(\ell),1/\ell)$-regular, and
\item[(iii)] $\mathcal{H}$ is $\delta_3$-regular w.r.t. $P_{\alpha\beta\gamma}^{hij}$ for all but at most $\delta_3t^3\ell^3$ triads
$$P_{\alpha\beta\gamma}^{hij}=P_{\alpha}^{hi}\cupdot P_{\beta}^{hj}\cupdot P_{\gamma}^{ij}=(V_h\cupdot V_i\cupdot V_j, E_{\alpha}^{hi}\cupdot E_{\beta}^{hj}\cupdot E_{\gamma}^{ij}),
$$
with $1\leq h < i < j \leq t$ and $\alpha, \beta, \gamma \in [\ell]$.
\end{enumerate}
\end{prop}
\noindent

\bigskip

To state the Counting Lemma for $3$-graphs (which typically complements the Regularity Lemma in the regularity method), consider the following setup.

\bigskip

\noindent\textbf{Setup A}. Let $k, m \in \mathbb{N}$ and $\delta_3, d_2, \delta_2 > 0$ be given. Suppose that
\begin{enumerate}
\item $V = V_1 \cup \dots \cup V_k, |V_1|=\dots=|V_k| = m$, is a partition of $V$.
\item $P =\bigcup_{1\leq i<j \leq k}P^{ij}$ is a $k$-partite graph, with vertex set $V$ and $k$-partition above, where all $P^{ij} = P[V_i, V_j],$
$1\leq i<j \leq k$, are $(\delta_2, d_2)$-regular.
\item $\mathcal{H} = \bigcup_{1\leq h<i<j\leq k}\mathcal{H}^{hij} \subseteq T(P)$ is a $k$-partite $3$-graph, with vertex set $V$ and $k$-partition above, where all
$\mathcal{H}^{hij} = \mathcal{H}[V_h, V_i, V_j]$, $1\leq h<i<j \leq k$, are $(\delta_3,d_{hij})$-regular
%r
with respect to the triad $P^{hi} \cup P^{ij} \cup P^{hj}$, for some constant $0\leq d_{hij}\leq 1$.
\end{enumerate}
\noindent
Now, first recall the counting lemma for $2$-graphs in the following version.
\begin{prop}[Counting lemma for graphs]\label{prop:2graphcounting}
For every integer $k\geq 3$ and positive constants $\gamma>0, d_2>0$ there exist $\delta_2=\delta_2(k,\gamma,d_2)>0$ and $m_0=m_0(k,\gamma,d_2,\delta_2)\in \mathbb{N}$ such that with these constants if a graph $P$ is as in 1. and 2. of Setup A and $m\geq m_0$, then
$$
(1-\gamma)d_2^{\binom{k}{2}}m^k\leq \N_k(P)\leq (1+\gamma)d_2^{\binom{k}{2}}m^k\;.
$$
\end{prop}
\noindent
Now, we state the Counting Lemma for $3$-graphs~\cite[Corollary 2.3]{NPRS}.
\begin{prop}[Counting Lemma for $3$-graphs]\label{prop:counting}
%\emph{(Counting lemma~\cite[Corollary 2.3]{NPRS})},
For every integer $k\geq 4$, and positive constants $\gamma, d_3 > 0$ there exists
$\delta_3 =\delta_3(k,\gamma, d_3)> 0$ so that for all $d_2 > 0$ there exist
%integer $r$ and
$\delta_2=\delta_2(k,\gamma,d_3,\delta_3,d_2) > 0$ and $m_0=m_0(k,\gamma,d_3,\delta_3,d_2,\delta_2)\in \mathbb{N}$ so that with these constants, if $\mathcal{H}$ and $P$ are as in Setup A, with $m\geq m_0$ and $d_{hij}\geq d_3$ for all $1\leq h<i<j\leq k$, then
$$
\N_k(\mathcal{H})\geq (1 - \gamma)d_3^{\binom{k}{3}}d_2^{\binom{k}{2}}m^k.
$$
%\leq (1 + \gamma)d_3^{\binom{k}{3}}d_2^{\binom{k}{2}}n^k\;
\end{prop}
\noindent
Combining Propositions~\ref{prop:counting} and~\ref{prop:2graphcounting} gives
\begin{cor}\label{cor:comparing}
For every $k\geq 4$, and positive constants $\gamma, d_3 > 0$ there exists
$\delta_3 =\delta_3(k,\gamma, d_3)> 0$ so that for all $d_2 > 0$ there exist
%integer $r$ and
$\delta_2=\delta_2(k,\gamma,d_3,\delta_3,d_2) > 0$ and $m_0=m_0(k,\gamma,d_3,\delta_3,d_2,\delta_2)\in \mathbb{N}$ so that with these constants, if $\mathcal{H}$ and $P$ are as in Setup A, with $m\geq m_0$ and $d_{hij}\geq d_3$ for all $1\leq h<i<j\leq k$, then
$$\N_k(\mathcal{H})\geq (1 - \gamma)d_3^{\binom{k}{3}} \N_k(P).
$$
\end{cor}

\begin{proof}
%Combine with parameter $\gamma/2$.
Let $\delta_3=\delta_3(k,\gamma/2,d_3)$ be the constant from Proposition~\ref{prop:counting}, and take $\delta_2$ to be the minimum
of $\delta_2(k,\gamma/2,d_2)$ from Proposition~\ref{prop:2graphcounting} and $\delta_2(k,\gamma/2,d_3,\delta_3,d_2)$ from Proposition~\ref{prop:counting}.
Let $m_0$ be the larger $m_0$ from these two propositions. Then, for $m\geq m_0$ we have
$$
\N_k(\mathcal{H})\geq (1 -\frac{\gamma}{2})d_3^{\binom{k}{3}}d_2^{\binom{k}{2}}m^k\geq (1 - \gamma)d_3^{\binom{k}{3}}(1+\frac{\gamma}{2})d_2^{\binom{k}{2}}m^k\geq (1 - \gamma)d_3^{\binom{k}{3}}\N_k(P).
$$
\end{proof}
\noindent

As another corollary of Proposition~\ref{prop:counting}, we obtain the following criterion for finding an \emph{induced} copy of a fixed $3$-graph $\F$.

\begin{cor}[Induced embedding lemma]\label{cor:embedding}
For every $3$-graph $\F$ with $|V(\F)|=k\geq 4$, and every constant $1>d_3 > 0$ there exists
$\delta_3=\delta_3(\F,d_3) > 0$ so that for all $d_2 > 0$ there exist
%integer $r$ and
$\delta_2=\delta_2(\F,d_3,\delta_3,d_2) > 0$ and $m_0=m_0(\F,d_3,\delta_3,d_2,\delta_2)\in \mathbb{N}$ so that, with these parameters, if $\mathcal{H}$ and $P$ are as in Setup A, with $m\geq m_0$ and $d_3 \leq d_{hij}\leq 1-d_3$ for all $1\leq h<i<j\leq k$, then $\mathcal{H}$ contains vertices $v_1,\dots, v_k$ with $v_i\in V_i$ for all $1\leq i \leq k$, such that
\begin{itemize}
\item $P[v_1,\dots, v_k]$ is a $k$-clique, and
\item $\mathcal{H}[v_1,\dots, v_k]$ is an induced copy of $\F$.
\end{itemize}
\end{cor}
\begin{proof}
Choose an arbitrary labeling $\varphi:[k] \rightarrow V(\F)$, and apply Proposition~\ref{prop:counting} to the $3$-graph $\mathcal{H}_{*}$ defined as follows
\begin{itemize}
\item $V(\mathcal{H}_*)=V(\mathcal{H})$, and
\item $\mathcal{H_*} = \bigcup_{1\leq h<i<j\leq k}\mathcal{H_*}^{hij}$ where all $\mathcal{H_*}^{hij} = \mathcal{H_*}[V_h, V_i, V_j]$, $1\leq h<i<j \leq k$ are defined by $\mathcal{H_*}^{hij}=\mathcal{H}^{hij}$ if $\{\varphi(h),\varphi(i),\varphi(j)\}\in \F$ and $\mathcal{H_*}^{hij}=\overline{\mathcal{H}}^{hij}=T(P[V_h,V_i,V_j])\setminus E(\mathcal{H})$ otherwise.
%$\overline{\mathcal{H}}^{hij}$ (the tripartite complement of $\mathcal{H}^{hij}$)
\end{itemize}
Observe that $\overline{\mathcal{H}}^{hij}$ is $(\delta_3, 1-d_{hij})$-regular with respect to its triad. Therefore $\mathcal{H}_*$ satisfies all the requirements of Proposition \ref{prop:counting}, implying that it contains a copy of $K^{(3)}_k$, which corresponds to an induced copy of $\F$ in $\mathcal{H}$.
\end{proof}

In what follows, we write $d(\G)=e(\G)/\binom{|\G|}{3}$ for the \emph{density} of a $3$-graph $\G$.
We shall also need the following rudimentary estimate for $3$-graph Tur\'{a}n densities.

\begin{claim}\label{cl:turan}
If $d(\G)>1-\binom{k}{3}^{-1}$ then $\G$ contains a $k$-clique $K_k^{(3)}\subseteq \G$.
\end{claim}
\begin{proof}Consider a random $k$-vertex subset of $V(\G)$.
By the union bound, the probability that one of the $\binom{k}{3}$ edges is missing is less than 1.
\end{proof}

\subsection{Proof of Lemma~\ref{lem:main}}
\begin{proof}%[Proof of Lemma~\ref{lem:main}]
%Suppose that $\G$ contains no induced copy of $\F$.

Given $\F$, $k$ and $\eta$, put $f:=|V(\F)|$, and let us pick $s$ and $\lambda$ satisfying
\begin{equation}\label{eq:const}
s=R^{(3)}(k,k,f)   ~~~~\mbox{and}~~~~  (1-\lambda)^{\binom{k}{3}+1}= 1-\eta\;,
\end{equation}
where $R^{(3)}(k,k,f)$ is the $3$-color Ramsey\footnote{That is, the smallest $R$ such that every edge-coloring of the complete $3$-graph $K_R^{(3)}$ with $3$ colors \emph{red}, \emph{blue}, and \emph{green} is guaranteed to contain either a red or a blue $K_k^{(3)}$, or a green $K_{f}^{(3)}$. see e.g.~\cite{Byellow}.} function of $3$-graphs.
Let us now set the stage for an application of Proposition~\ref{prop:reglemma} (the $3$-graph regularity lemma) by defining the following parameters.
\begin{equation}\label{eq:delta3def}
\delta_3= \min\left\{\frac{\lambda}{3},~\frac{s^{-3}}{10},~\delta_3(k,\lambda, 1-\lambda),~\delta_3(\F,\frac{\lambda}{3}), \right\},
\end{equation}
where $\delta_3(k,\lambda, 1-\lambda)$ is as stated in Corollary~\ref{cor:comparing}, %with $k=k$, $d_3=1-\lambda$ and $\gamma=\lambda$
and $\delta_3(\F,\lambda/3)$ is as stated in Corollary~\ref{cor:embedding}. We define the function $\delta_2:\mathbb{N}\rightarrow (0,1]$
to satisfy for every integer $\ell$
$$
\delta_2(\ell)=\min\{\delta_2(k,\lambda,1-\lambda,\delta_3,1/\ell),~\delta_2(\F,\lambda/3,\delta_3,1/\ell)\}\;,
$$
where $\delta_2(k,\lambda,1-\lambda,\delta_3,1/\ell)$ and $\delta_2(\F,\lambda/3,\delta_3,1/\ell)$ are as stated in Corollary~\ref{cor:comparing} and Corollary~\ref{cor:embedding}, respectively. Finally, we set
$$
t_0=s\;.
$$
Given an induced $\F$-free $3$-graph $\G$, we apply Proposition~\ref{prop:reglemma} with the above constants/function. We obtain partitions
$V(\G)=V_0\cupdot\dots\cupdot V_t$ and $\mathcal{P}^{ij}$ as stated in Proposition~\ref{prop:reglemma}. We will assume henceforth that $n/(2t)>m_0$, where $m_0$ is the maximum between the values of $m_0$ claimed in Proposition~\ref{prop:counting} and Corollary~\ref{cor:embedding} (with respective parameters). Observe that from the way we chose the parameters above,
it follows that the parameters $t$ and $\ell$ of the partition
we obtain through Proposition~\ref{prop:reglemma} can be bounded from above by functions of $\F,k,\eta$. Since the constants $c',\epsilon$
we will obtain below will depend only on $t,\ell$ they will indeed depend only on $\F,k,\eta$ as asserted in the statement of the lemma.

Let $P^{ij}_{\alpha}$ be the collection of bipartite graphs from item $(ii)$ of Proposition~\ref{prop:reglemma}.
For each $1\leq i < j \leq t$ select $\alpha_{ij}\in [\ell]$ independently and uniformly at random, set $P^{ij}=P^{ij}_{\alpha_{ij}}$ and denote by $P^{hij}$ the triad obtained from $P^{hi}$,
$P^{hj}$ and $P^{ij}$. Then the expected proportion of triads $P^{hij}$ such that $\mathcal{H}$ is $\delta_3$-regular with respect to  $P^{hij}$, is at least $1-10\delta_3$. Thus, there must exist a choice of $\alpha_{ij}$'s satisfying this property.
To simplify the presentation, let us fix one such choice of $\alpha_{ij}$'s and use $P^{ij}$ to denote the corresponding collection of bipartite graphs
defined by this choice.
Define the $3$-uniform \emph{cluster-hypergraph} $\R$ as follows: let $V(\R)=[t]$ and $\{h,i,j\} \in E(\R)$ if $\mathcal{H}$ is $\delta_3$-regular with respect to the triad $P^{hij}$. By the above we have $d(\R)\geq 1-10\delta_3 >1-\binom{s}{3}^{-1}$, so by Claim~\ref{cl:turan} there exists a collection of indices $1\leq i_1<\dots<i_s\leq t$ such that all triads $P^{i_pi_qi_r}$ are $\delta_3$-regular. Without loss of generality, let us assume that
$i_1<\dots<i_s$ are just $1,\ldots,s$.

Recaping, what we have at this point is $s$ vertex sets $V_{1},\ldots,V_s$ of size $(n-|V_0|)/t$ each, and a collection of $\binom{s}{2}$ bipartite graphs $P=\bigcup P^{ij}$, where $P^{ij}$ is $(\delta_2(\ell),1/\ell)$-regular (between $V_i$ and $V_j$). Furthermore, for every $h < i < j$, the $3$-graph $\G$ is $\delta_3$-regular with respect to the triad $P^{hij}$ formed by $P^{hi}$,
$P^{hj}$ and $P^{ij}$. In other words, we are at Setup A with
\begin{equation*}
s,~~~~~m=(n-|V_0|)/t,~~~~~ P=\bigcup P^{ij},~~~~~ \G\cap T(P)
\end{equation*}
playing the role of $k$, $m$, $P$ and $\mathcal{H}$, respectively, and $\delta_2$ and $\delta_3$ as defined above.

Now define an edge-coloring $\Gamma$ of the complete $3$-graph on $[s]$ as follows.
Color the edge $\{h,i,j\}\in [s]^{(3)}$ \emph{red} if the corresponding triad $P^{hij}$ satisfies  $d(\G|P)>1-2\lambda/3$, where $\lambda>0$
was chosen in (\ref{eq:const}). Color $P$ \emph{blue} if $d(\G|P)<2\lambda/3$, and color $P$ \emph{green} otherwise.
By our choice of $s$ in (\ref{eq:const}), one of the following assertions must hold.
\begin{itemize}
\item[(g)]
There exist $f$ indices $\{j_1,\dots, j_{f}\} \subseteq [s]$ such that every $\Gamma(j_p,j_q,j_r)$ is green for all $1\leq p<q<r\leq f$. Without loss of generality, we may assume that these indices are simply $1,\dots, f$.
\item[(r)] There exist $k$ indices (wlog) $\{1,\dots, k\}$ such that every  $\Gamma(p,q,r)$ is red.
\item[(b)] There exist $k$ indices (wlog) $\{1,\dots, k\}$ such that every  $\Gamma(p,q,r)$ is blue.
\end{itemize}
Due to our choice of parameters, we can now invoke the counting lemmas. In Case (g) we apply Corollary~\ref{cor:embedding} (note that, by Remark~\ref{rem:regdens} and~\eqref{eq:delta3def}, $\G$ is $(\delta_3,d_{pqr})$-regular with respect to $P^{pqr}$ for some $\lambda/3\leq d_{pqr}\leq 1-\lambda/3$) to deduce that $\G$ contains vertices $v_1,\dots v_{f}$ such that $v_p\in V_{p}$ for each $1\leq p\leq f$ and $\G\cap T(P)$ induces a copy of $\F$ on $v_1,\dots v_{f}$. It remains to show that $\G$ itself also induces a copy of $\F$ on $v_1,\dots v_{f}$.

For a given $1\leq p<q<r\leq f$, if $\{v_p,v_q,v_r\}\in E(\G\cap T(P))=E(\G)\cap T(P)$, then clearly also $\{v_p,v_q,v_r\}\in E(\G)$. On the other hand, since, by Corollary~\ref{cor:embedding}, $P[v_1,\dots,v_f]$ is a clique, for each $1\leq p<q<r\leq f$ we have $\{v_p,v_q,v_r\}\in T(P)$. Thus, if $\{v_p,v_q,v_r\}\notin E(\G)\cap T(P)$, then $\{v_p,v_q,v_r\}\notin E(\G)$. We conclude that
 $\G[v_1,\dots,v_{f}]=(\G\cap T(P))[v_1,\dots,v_{f}]$, and thus $\G[v_1,\dots,v_{f}]$ is indeed an induced copy of $\F$, a contradiction. So, case (g) cannot occur.

So, suppose that Case (r) holds.
Then, applying Proposition~\ref{prop:2graphcounting}, we obtain that the graph
\begin{equation}\label{eq:Pk}
P^*=\bigcup_{1\leq p<q\leq k} P^{pq}
\end{equation}
satisfies
\begin{equation}\label{eq:Pstareq}
\N_k(P^*)\geq \frac{1}{2}m^k\left(\frac{1}{\ell}\right)^{\binom{k}{2}}\;.
\end{equation}
Furthermore, since by Remark~\ref{rem:regdens} and~\eqref{eq:delta3def}, $\G$ is $(\delta_3,d_{pqr})$-regular with respect to $P^{pqr}$ for some $d_{pqr}\geq 1-\lambda$,  Corollary~\ref{cor:comparing} gives
\begin{equation*}
\N_k(\G\cap T(P^*))\geq (1-\lambda)\cdot(1-\lambda)^{\binom{k}{3}}\N_k(P^*)=(1-\eta)\N_k(P^*)\;.
\end{equation*}
Thus, we have found $k$ vertex sets $V_{1},\dots V_{k}$ of size $m \geq c'n$ and a $k$-partite graph $P^*$ on these sets, which is $(\epsilon,\eta)$-dense with respect to $\G\cap T(P^*)\subseteq \G$, where $\epsilon=\frac{1}{2}(\frac{1}{\ell})^{\binom{k}{2}}$.

Similarly, in Case (b), considering $P^*$ as in~\eqref{eq:Pk} and the $3$-graph $\hat{\G}=T(P^*)\setminus \G$ (note that $d(\hat{\G}| P^*)>1-2\lambda/3$) we obtain that~\eqref{eq:Pstareq} holds as before, and, by Corollary~\ref{cor:comparing},
%by Proposition~\ref{prop:2graphcounting},
\begin{equation*}
\N_k(\hat{\G})\geq (1-\lambda)\cdot(1-\lambda)^{\binom{k}{3}}\N_k(P^*)=(1-\eta)\N_k(P^*)\;.
\end{equation*}
Thus, again, for $\epsilon=\frac{1}{2}(\frac{1}{\ell})^{\binom{k}{2}}$ we obtain the sets $V_{1},\dots V_{k}$ of size $m \geq c'n$ and the graph $P^*$ on these sets, which is $(\epsilon,\eta)$-dense with respect to $\hat{\G}\subseteq \overline{\G}$.
\end{proof}

\section{Proofs of Theorems \ref{thm:example} and \ref{thm:stepup}}\label{sec:examples}

In this section we prove Proposition~\ref{prop:extend} and deduce Theorems~\ref{thm:example} and~\ref{thm:stepup} as corollaries.

We say that an $r$-graph $\F$ is \emph{extendable} if any two vertices of $\F$ are contained in some edge and $\F$ contains a copy of $K^{(r)}_{r+1}$ (complete $r$-graph on $r+1$ vertices). Note that for $r=2$ the extendable graphs are precisely the cliques. Now we are ready to state Proposition~\ref{prop:extend}.

\begin{prop}\label{prop:extend}
For every $r \geq 3$ and an extendable $(r-1)$-graph $\F$
there exists an $r$-graph $\F^+$ and constant $c_0=c_0(r,\F)>0$ and $c_1=c_1(r,\F)>0$ such that
$$
R_{r,\F^+}(t) \geq \left(R_{r-1,\F}(c_0t)\right)^{c_1t}\;.
$$
\end{prop}

In order to construct the non-universal $r$-graph in the above proposition from a non-universal $(r-1)$-graph we will adapt an idea from \cite{CFShyp}.
Given an extendable $(r-1)$-graph $\F$, an $(r-1)$-graph $\G=(V,E)$ and an edge-labeling $\phi:\binom{[N]}{2}\rightarrow [|V|]$ for some integer $N\geq |V|$, define a red-blue coloring $\chi:\binom{[N]}{r}\rightarrow \{\text{red}, \text{blue}\}$ as follows.
For any $1\leq a_1<\dots<a_r\leq N$ we define $\chi(\{a_1,\dots a_r\})=\text{red}$ if for $i=2,\dots,r$ all $\phi(a_1,a_i)$ are distinct and $\{\phi(a_1,a_2)\dots, \phi(a_1,a_r)\}\in E(\G)$, and otherwise $\chi(\{a_1,\dots,a_r\})=\text{blue}$. Set $\G^+=\G^{+}(\G,\phi)$ to be the $r$-graph of all red edges in $\chi$.

The first step towards proving Proposition \ref{prop:extend} is to construct the graph $\F^+$.

\begin{lemma}\label{lem:induceup}
Let $r\geq 3$. Then for every extendable $(r-1)$-graph $\F$ there exists an $r$-graph $\F^+$ satisfying the following.
For any induced $\F$-free $(r-1)$-graph $\G$
and any labelling $\phi$ the $r$-graph $\G^+=\G^+(\G,\phi)$ is induced $\F^+$-free.
\end{lemma}

\begin{proof}
To construct $\F^+$ we proceed as follows. We first define an \emph{ordered} $r$-graph $\F^*$ such that $\G^+$, when viewed as an ordered $r$-graph under the natural ordering on $[N]$, does not contain an induced copy of $\F^*$. Put $f=|V(\F)|$, and by choosing an arbitrary ordering suppose that $V(\F)=[f]$. Now, let $V(\F^*)=\{0\}\cup [f]$, and $E(\F^*)=\{\{0\}\cup J:J\in E(\F)\}\cup \binom{[f]}{r}$.

Let us now show that $\G^+$ indeed does not contain an induced copy of $\F^*$ as an ordered subgraph. Suppose for a contradiction that there is an order preserving mapping $\tau:\{0\}\cup [f] \rightarrow [N]$ such that $\{\tau(j_1),\dots \tau(j_r)\}\in E(\G^+)$ if and only if $\{j_1,\dots,j_r\} \in E(\F^*)$.
Denote $\tau(0)=a$ and $\tau(i)=b_i$ for each $1\leq i\leq f$.
%, and note that the link hypergraph\footnote{define?} of $\tau(\F^*)\subseteq \G^+$ with respect to $a\in \tau(\F^*)$ is isomorphic to $\F$.

Suppose that $\phi(a,b_i)=\phi(a,b_k)$ for some $1\leq i<k\leq f$. Then by the definition of $\chi$ no edge of $\G^+$ contains $\{a,b_i, b_k\}=\{\tau(0),\tau(i),\tau(k)\}$. Since $\tau$ is an isomorphism, no edge of $\F^*$ contains $\{0,i,k\}$. This in turn implies that $\{i,k\}$ is not contained in any edge of $\F$,
%Moreover, since for any vertex set $\{j_1,\dots, j_{r-2}\} \subseteq [f]\setminus \{i,k\}$ we have $\{\phi(a,\tau(j_1)),\dots , \phi(a,\tau(j_{r-2})),\phi(a,b_i)\} = \{\phi(a,\tau(j_1)),\dots , \phi(a,\tau(j_{r-2})),\phi(a,b_k)\}$ and since $a=\tau(0)$, $b_i=\tau(i)$ and $b_k=\tau(k)$, we must have
%$$\chi(\{\tau(0),\tau(j_1),\dots,\tau(j_{r-2}),\tau(i)\})=\chi(\{\tau(0),\tau(j_1),\dots,\tau(j_{r-2}),\tau(k)\}).
%$$
%Thus, $\{0,j_1,\dots,j_{r-2},i\}\in E(\F^*)$ if and only if $\{0,j_1,\dots,j_{r-2},k\}\in E(\F^*)$. Equivalently, we have $\{j_1,\dots,j_{r-2},i\}\in E(\F)$ if and only if $\{j_1,\dots,j_{r-2},k\}\in E(\F)$. This, and the above fact that $\{i,k\}$ is not contained in any edge of $\F$ imply that $i$ and $k$ are twins in $\F$
a contradiction to the assumption that $\F$ is extendable. Hence, all labels $\phi(a,b_i)$ for $1\leq i\leq f$ must be distinct.

Now, take an arbitrary $I=\{j_1,\dots,j_{r-1}\}\subseteq [f]$, and put $a_{i}=\tau(j_i)$ for each $1\leq i\leq r-1$.
If $I\in E(\F)$, then we have $\{0\}\cup I\in E(\F^*)$, thus
$\{a,a_1,\dots, a_{r-1}\}=\{\tau(0),\tau(j_1),\dots \tau(j_{r-1})\}\in E(\G^+)$. Hence, by the definition of $\chi$, we must have $\{\phi(a,a_1),\dots, \phi(a,a_{r-1})\}\in E(\G)$. On the other hand, if $I\notin E(\F)$, then $\{a,a_1,\dots, a_{r-1})\}\notin E(\G^+)$, which similarly means $\chi(\{a,a_1,\dots, a_{r-1}\})=\text{blue}$. Since we have established that all $\phi(a,\tau(j_i))$ are distinct, we deduce that $\{\phi(a,a_1),\dots \phi(a,a_{r-1})\}\notin E({\G})$.

Therefore the distinct labels $\{\phi(a,b_1),\dots \phi(a,b_f)\}\subseteq V(\G)$ form a copy of $\F$ as an induced subgraph of $\G$, in contradiction to the assumption that $\G$ is induced $\F$-free. We conclude that $\G^+$ does not contain an ordered induced copy of $\F^*$.

Now, define $\F^+$ to be the $r$-graph such that every ordering of it contains $\F^*$ as an induced ordered subgraph. Such a graph always exists by a theorem of Alon, Pach and Solymosi~\cite[Theorem 3.5]{APS}. Then $\G^+$ does not contain an induced copy of $\F^+$, for otherwise it would contain an ordered induced copy of $\F^*$.
\end{proof}

We will later need the following simple observation.

\begin{claim}\label{rem:cliqueplus}
In the natural case when $\F=K_{s-1}^{(r-1)}$ for some $s>r$, one can take $\F^+=K_s^{(r)}$.
\end{claim}

\begin{proof}
$\F$ is extendable, $\F^*=K_s^{(r)}$, and since all orderings of $\F^*$ are indistinguishable, we can take $\F^+=\F^*=K_s^{(r)}$.
\end{proof}

Having defined $\F^+$ we need to find a labelling $\phi$, such that if neither $\G$ nor its complement contain a large complete $(r-1)$-partite subgraph then $\G^+(\G,\phi)$ will also have this property. This is done in the following technical claim whose proof is deferred to the end of the section.

\begin{claim}\label{lem:edgel}
For every integer $r\geq 3$ and constant $\alpha>1$ there exists a constant $c_1=c_1(r,\alpha)>0$ and an integer $t_0=t_0(r,\alpha)$ as follows.
Let $\beta=\frac{2}{1-1/\alpha}$, $c_0=\frac{1}{\beta r^2}$, and suppose that the integers $t, n$ and $N$ are such that $t\geq t_0$,
$n\geq (c_0t)^\alpha$ and $N = n^{c_1t}$. Then there exists a labeling $\phi: \binom{[N]}{2}\rightarrow [n]$ with the following property:
There do not exist disjoint sets $A_1,\dots, A_r\subseteq [N]$ with $|A_1|=\dots=|A_r|={t}/{r}$ such that for every $a\in A_1$ we have
\begin{equation}\label{eq:ABC}
|\{\phi(a,a_i):a_i\in A_i\}| <\frac{t}{\beta r}\;, \mbox{~~~for some~~~}  2 \leq i \leq r
\end{equation}
\end{claim}
\noindent

The last ingredient we will need for the proof of Proposition \ref{prop:extend} is the following
claim whose proof is a routine application of the probabilistic deletion method
and is thus deferred to the end of the section.

\begin{claim}\label{lem:gnp}
Let $\F$ be an extendable $r$-graph. Then $R_{r,\F}(t)\geq t^\alpha$ for some $\alpha=\alpha_r>1$.
\end{claim}
\begin{proof}[Proof of Proposition~\ref{prop:extend}]
Given an extendable $(r-1)$-graph $\F$, take $\F^+$ to be the $r$-graph from lemma~\ref{lem:induceup}. Let $\alpha=\alpha_{r-1}$ be the constant from Claim~\ref{lem:gnp}, and suppose that $t\geq t_0$, where $t_0=t_0(r,\alpha)$ is as defined in Claim~\ref{lem:edgel} (due to flexibility in choosing the constants $c_0$ and $c_1$, it suffices to prove the statement of Proposition~\ref{prop:extend} for large $t$). Let $\beta=\frac{2}{1-1/\alpha}$, $c_0=\frac{1}{\beta r^2}$, and let $c_1=c_1(r,\alpha)$ be as in Claim~\ref{lem:edgel}. Let $n=R_{r-1,\F}(c_0t)$, and note that, since $\F$ is extendable, Claim~\ref{lem:gnp} gives $n\geq (c_0t)^\alpha$. Let $\G$ to be an $(r-1)$-graph on $n$ vertices which is induced $\F$-free and such that $\G$ and $\overline{\G}$ do not contain a $K^{(r-1)}_{c_0t,\dots,c_0t}$.
Finally, let $N=n^{c_1t}$, let $\phi:\binom{N}{2}\rightarrow [n]$ be a labelling as guaranteed by Claim~\ref{lem:edgel}, and consider the $r$-graph $\G^+=\G^+(\G,\phi)$.
%\begin{lemma}\label{lem:partiteup}
%Suppose that $\G$ is an $(r-1)$-graph on $n$ vertices such that both $\G$ and $\overline{\G}$ do not contain a copy of $K^{(r-1)}_{c_0t,\dots,c_0t}$. Then, with $\phi$ as above and $\G^+=\G^+(\G,\phi)$, both $\G^+$ and $\overline{\G^+}$ do not contain a copy of $K^{(r)}_{t,\dots,t}$.
%\end{lemma}

We now claim that $\overline{\G^+}$ contains no copy of $K^{(r)}_{t,\dots,t}$; the fact that $\G^+$ does not contain a $K^{(r)}_{t,\dots,t}$ can be deduced analogously. So, suppose for a contradiction that $\overline{\G^+}$ contains a copy of $K^{(r)}_{t,\dots ,t}$ on vertex classes $Q_1,\dots , Q_r \subseteq V(\G^+)$. Let $A'$ be the first $t$ vertices of $\bigcup_{i=1}^r Q_i$ in the natural ordering of $[N]$. By the pigeonhole principle, some $t/r$ of them will be in the same class, say in $Q_1$; let $A_1\subseteq A'$ be the set of these $t/r$ vertices. By our choice of $A'$, for $i=2,\dots, r$ there will be disjoint sets $A_i \subseteq Q_i$ with $|A_i|=t/r$ such that for every $a_1\in A_1,$ and $a_i\in A_i$ we have $a<a_i$. Note that we do {\em not} make any claims regarding the order of vertices inside $A_2\cup \dots\cup A_r$.

By the definition of $\phi$ and Claim~\ref{lem:edgel} there must exist some $a\in A$ such that, with
%Now, fix some $a\in A$ and let
$\Phi(A_i)=\{\phi(a,a_i):a_i\in A_i\}$, we have $\min\{|\Phi(A_i)|\colon 2\leq i\leq r\}\geq \frac{t}{\beta r}$. Now, taking one representative vertex for each color in each $\Phi(A_i)$, we find sets $U_1\subseteq A_i$ for $2\leq i\leq r$ with $|U_i|=\frac{t}{\beta r^2}=c_0t$ such that all labels in $\bigcup_{i=2}^r\bigcup_{u\in U_i} \phi(a,u)$ are distinct. However, since for each $(u_2,\dots,u_r)\in U_2\times\dots\times U_r$ we have $\chi(\{a,u_2,\dots,u_r\})=\text{blue}$, we obtain that $\overline{\G}[\Phi(U_2),\dots,\Phi(U_r)]$ forms a complete $(r-1)$-partite graph with parts of size $c_0t$, in contradiction to the assumption. Hence, $\overline{\G^+}$ contains no copy of $K^{(r)}_{t,\dots ,t}$, as claimed.

To summarize, the $r$-graph $\G^+$ has $N=n^{c_1 t}$ vertices, by Lemma~\ref{lem:induceup} it is induced $\F^+$-free, and neither $\G^+$ nor its complement contain a copy of $K^{(r)}_{t,\dots,t}$. Therefore,
$$
R_{r,\F^+}(t)\geq N = (R_{r-1,\F}(c_0t))^{c_1 t}\;,
$$
as needed.
\end{proof}
\noindent
As corollaries of Proposition~\ref{prop:extend} we obtain Theorems~\ref{thm:example} and~\ref{thm:stepup}.

\begin{proof}[Proof of Theorem~\ref{thm:example}]
Apply Proposition~\ref{prop:extend} for $r=3$ and $\F$ being the triangle, that is, the complete $2$-graph on $3$ vertices; note that $\F$ is extendable. By Claim~\ref{rem:cliqueplus} we can take $\F^+=K_4^{(3)}$. Hence, by Proposition~\ref{prop:extend}, we have $R_{3,\F^+}(t)\geq (c_0t)^{\alpha c_1 t}=t^{\Omega(t)}$, as claimed.
\end{proof}

\begin{proof}[Proof of Theorem~\ref{thm:stepup}]
By adding to $\F$ a set of $r+1$ new vertices of full degree, we obtain an extendable graph $\F'$. Applying Proposition~\ref{prop:extend} to $\F'$ we obtain a graph $\F^+$ and the constants $c_0(r,\F), c_1(r,\F)>0$ as stated therein. Then, taking $c(r,\F)=\min\{c_0,c_1\}>0$, we obtain
$$
R_{r,\F^+}(t) \geq \left(R_{r-1,\F'}(c_0t)\right)^{c_1t}\geq \left(R_{r-1,\F}(ct)\right)^{ct}\;,
$$
where the second inequality relies on the fact that $\F'$ contains an induced copy of $\F$.
\end{proof}

We end this section with the proofs of Claims \ref{lem:edgel} and \ref{lem:gnp}.

\begin{proof}[Proof of Claim~\ref{lem:edgel}]
Set
$$
\gamma=\frac{1}{2}(1-\frac{1}{\alpha}-\frac{1}{\beta})>0 \ , \ c_1=\frac{\gamma}{r^2}=\frac{1}{2r^2}(1-\frac{1}{\alpha}-\frac{1}{\beta}) \ \  \text{and} \ \ t_0=c_1^{-\frac{1+\alpha\gamma}{\alpha \gamma}}\;.
$$
and suppose that $t, n$ and $N$ are as in the statement of the lemma. Define $\phi: \binom{[N]}{2}\rightarrow [n]$ to be an edge-labeling, where each element of $\binom{[N]}{2}$ is independently assigned a color between $1$ and $n$ uniformly at random. For disjoint sets $A_1,\dots, A_r \subseteq [N]$ of size ${t}/{r}$ each, let $X_{A_1,\dots, A_r}$ be the event that every $a\in A_1$ satisfies~\eqref{eq:ABC}.
Consider some fixed disjoint $A_1, \dots, A_r$ of size ${t}/{r}$. Since for different $a\in A_1$ the edges between $a$ and $A_2\cup\dots \cup A_r\subseteq [N]\setminus A_1$ are assigned the values of $\phi$ independently, we obtain
\begin{align*}
\Prob(X_{A_1, \dots, A_r})&\leq \left[\sum_{i=2}^{r}\binom{n}{\frac{t}{\beta r}}\left(\frac{\frac{t}{\beta r}}{n}\right)^{|A_i|}\right]^{|A_1|} \leq \left[r\binom{n}{\frac{t}{\beta r}}\left(\frac{t}{n}\right)^{\frac{t}{r}}\right]^{\frac{t}{r}}\;.
\end{align*}
By the union bound, the probability that some event $X_{A_1,\dots,A_r}$ holds is at most
\begin{align*}
\sum_{A_1,\dots,A_r} \Prob\left( X_{A_1,\dots,A_r}\right)&\leq  {N}^{t}\left[r\binom{n}{\frac{t}{\beta r}}\left(\frac{t}{n}\right)^{\frac{t}{r}}\right]^{\frac{t}{r}}\leq \left(N^r n^{\frac{t}{\beta r}}\left(c_0^{-1}n^{\frac{1}{\alpha}-1}\right)^{\frac{t}{r}}\right)^{\frac{t}{r}}\\
&=\left(c_0^{-\frac{t}{r}}n^{c_1rt+\frac{t}{\beta r}+(\frac{1}{\alpha}-1)\frac{t}{r}}\right)^\frac{t}{r}=\left(c_0^{-\frac{t}{r}}n^{\frac{t}{r}(c_1r^2+\frac{1}{\beta}+\frac{1}{\alpha}-1)}\right)^\frac{t}{r}\\
&=(n^{-\gamma} c_0^{-1})^{(\frac{t}{r})^2}<(c_0^{1+\alpha \gamma}t^{\alpha\gamma})^{-(\frac{t}{r})^2}<1\;.
\end{align*}

We infer that with positive probability none of the events $X_{A_1,\dots,A_r}$ hold, implying that the required labeling exists.
\end{proof}

\begin{proof}[Proof of Claim \ref{lem:gnp}]
Since $R_{r,\F}(t)\geq R_{r,K^{(r)}_{r+1}}(t)$, it suffices to consider $K^{(r)}_{r+1}$. We will prove the equivalent statement, that for every sufficiently large $n$ there exists a $K^{(r)}_{r+1}$-free $r$-graph on $\Theta(n)$ vertices, which does not contain a complete or empty $r$-partite subgraph of polynomial size.

To this end, set $p=n^{-\frac{r-1/2}{r+1}}=n^{-1+\frac{3}{2(r+1)}}$ and let $\G$ be a  random $r$-graph $\G^{(r)}_{n,p}$ on $n$ vertices, where (as in random $2$-graphs) each $r$-tuple is selected as an edge independently with probability $p$.
Then the expected number of copies of $K^{(r)}_{r+1}$ in $\G$ will be $\binom{n}{r}p^{r+1}=\Theta(n^{1/2})$. Hence, by Markov's inequality the probability that $\G$ contains more than $n^{3/4}$ copies of $K^{(r)}_{r+1}$ is $o(1)$.

Next, for an integer $t$, the expected number of copies of the complete $r$-partite graph $K^{(r)}_{t,\dots,t}$ in $\G_{n,p}$ is
$$
\binom{n}{rt}\frac{1}{r!}\left[\prod_{k=r}^1\binom{kt}{t}\right]p^{t^r}\leq n^{rt}p^{t^r}=(n^rp^{t^{r-1}})^t\;,
$$
and a straightforward calculation shows that the latter expression is $o(1)$ when $t\geq 3$. Hence, by Markov's inequality, with high probability $\G$ will not contain a copy of $K^{(r)}_{3,\dots,3}$, let alone of $K^{(r)}_{t,\dots,t}$ for larger values of $t$.
Lastly, for $t=n^{\frac{3}{2r}}$ the expected number of copies of $K^{(r)}_{t,\dots,t}$ in $\overline{\G}$ is
\begin{equation}\label{eq:gnp}
\binom{n}{rt}\frac{1}{r!}\left[\prod_{k=r}^1\binom{kt}{t}\right](1-p)^{t^r}\leq n^{rt}e^{-pt^r}=(n^re^{-pt^{r-1}})^t\;.
\end{equation}
Again, a straightforward calculation shows that with this choice of $t$ the right hand side of~\eqref{eq:gnp} is $o(1)$, hence, again by Markov's inequality, with high probability there will be no copy of $K^{(r)}_{t,\dots,t}$ in $\overline{\G}$.

We deduce that for $t=n^{\frac{3}{2r}}$ with positive probability $\G\sim \G^{(r)}_{n,p}$ satisfies all three properties discussed above.
Removing a vertex from each $r+1$-clique thus results in a $K^{(r)}_{r+1}$-free $r$-graph on at least $n/2$ vertices satisfying the assertion of the lemma.
\end{proof}

\section{Concluding Remarks}\label{sec:concluding}

As we have mentioned in Section \ref{sec:intro}, the best known bound for Ramsey numbers of $3$-graphs imply that every $3$-graph
contains a homogenous set of size $\Omega(\log\log n)$. An intriguing open problem of Conlon, Fox and Sudakov \cite{CFS} is if for every 3-graph $\F$, every induced $\F$-free $\G$ contains a homogeneous set of size $\omega(\log\log n)$. As a (minor) step towards resolving this problem one would first
like to find general conditions guaranteeing that certain $\F$ satisfy this condition.
For example, it follows from the Erd\H{o}s--Rado \cite{ER} bound that this is the case if $\F$ is a complete $3$-graph. Let us now sketch an argument showing that a much broader family of graphs $\F$ have this property. We say that a $3$-graph $\F$ on $f$ vertices is {\em nice} if there is an ordering $\{1,\ldots,f\}$ of $V(\F)$ and an ordered $2$-graph $F$ on $\{1,\ldots,f\}$ so that for every $1 \leq i < j < k \leq f$ the triple $(i,j,k) \in E(\G)$ if and only if $(i,j) \in E(F)$.

\begin{proposition}
If $\F$ is nice, then every induced $\F$-free $3$-graph $\G$ on $n$ vertices has a homogeneous set of size $2^{\Omega(\sqrt{\log\log n})}$.
\end{proposition}

\begin{proof}[Proof (sketch):] Since $\F$ is nice there is a graph $F$ satisfying the condition in the above paragraph.
Pick an arbitrary ordering of $V(F)$ and let $F'$ be a graph so that every ordering of $V(F')$ has
an ordered induced copy of $F$. Such an $F'$ exists by \cite{RW}.

We recall that the Erd\H{o}s--Rado \cite{ER} scheme of bounding Ramsey numbers proceeds by gradually building an ordered set of vertices
$\{1,\ldots, p\}$, where $p=\Omega(\sqrt{\log n})$, along with an ordered graph $G$ on $[p]$ so that for every $1 \leq i < j < k \leq p$ the triple $(i,j,k) \in E(\G)$ if and only if $(i,j) \in E(G)$. Observe that $G$ is induced $F'$-free, as otherwise the definition of $G$, $F'$ and $F$ would imply that $\G$
has an induced copy of $\F$. By the Erd\H{o}s--Hajnal theorem \cite{EH} mentioned in Section \ref{sec:intro} we deduce that $G$ has a homogeneous set of size $2^{\Omega(\sqrt{\log\log n})}$, which
implies (using the definition of $G$) that $\G$ has a homogeneous set of the same size.
\end{proof}

Lastly, we state and prove Proposition \ref{prop:example} which extends R\"odl's~\cite{RodlUniv} example mentioned in Section \ref{sec:intro}
to arbitrary uniformities $r \geq 3$. We would like to reiterate our belief that the following upper bound is tight for all $r \geq 3$.
We remind the reader that we use $d(\G)=e(\G)/\binom{|\G|}{r}$ for the \emph{edge-density} of an $r$-graph $\G$.

\begin{proposition}\label{prop:example}
For every integer $r\geq 3$ there is an $r$-graph $\F$ with the following properties: $(i)$ for any $\epsilon>0$ there exists a constant $C=C(r,\epsilon)$ such that for arbitrarily large $n$ there exists an induced $\F$-free $r$-graph $\G$ on $n$ vertices, such that for every vertex set $W\subset V(\G)$ with $|W|=C(\log n)^{1/(r-2)}$ we have
\begin{equation}\label{eq:middle}
1/2-\epsilon\leq d(\G[W])\leq 1/2+\epsilon\;.
\end{equation}
$(ii)$ there exists a constant $C=C(r)$ such that for every large enough $n$ there is an induced $\F$-free $r$-graph $\G$ on $n$ vertices, that does not contain a copy of $K^{(r)}_{t,\dots,t}$ with $t=C(\log n)^{1/(r-2)}$.
\end{proposition}

\begin{proof}
We prove only item $(i)$, since the proof of $(ii)$ is identical.
We make use of the parity construction from~\cite{CG}. Let $V(\G)=[n]$, and consider an $(r-1)$-graph $\mathcal{H}\sim \G^{(r-1)}_{n,1/2}$, that is, each $(r-1)$-edge is selected randomly and independently with probability $1/2$. Then define $\G$ as follows: let $\{i_1,\dots,i_r\}\in E(\G)$ if and only if $e(\mathcal{H}[\{i_1,\dots,i_r\}])\equiv 0\ (\textrm{mod}\ 2)$.

It is easy to see that for each set $R$ of $r+1$ vertices in $\G$ we have $e(\G[R])\equiv r+1 \ (\textrm{mod}\ 2)$. Hence, $\G$ will not contain an induced copy of $\F$, where $\F$ is any $r$-graph with $e(\F)\equiv r \ (\textrm{mod}\ 2)$.\footnote{For even values of $r$ this argument can in fact be applied to \emph{any} $r$-graph $\F$ by considering either $\G$ or its complement}
Hence, to prove the lemma, it suffices to show that~\eqref{eq:middle} holds with  positive probability; this is achieved by a standard application of Azuma's inequality.

Consider a subset $A\subseteq [n]$, with $|A|=a$, and define the random variable $X=e(\G[A])$. Note that each subset $\{i_1,\dots,i_r\}\subseteq A$ forms an edge of $\G$ with probability $1/2$ -- this is a consequence of the basic identity
$$
\sum_{j=0}^r (-1)^{j}\binom{r}{j}=(1-1)^{r}=0\;.
$$
By linearity of expectation, this implies $\E(X)=\frac{1}{2}\binom{a}{r}$.

Now, choose an arbitrary ordering of the subsets of $A$ of size $r-1$ and consider the edge-exposure martingale $(X_t)_{t=0}^N$, where $N=\binom{a}{r-1}$, and $X_t$ is the expected value of $X$, after the first $t$ edges have been exposed. In particular, $X_0=\E(X)$ and $X_N=X$. Observe that exposing a new $(r-1)$-edge can change at most $a$ edges of $\G$. Hence, Azuma's inequality yields
$$ \Prob(|X_N-X_0|\geq \epsilon a^r)\leq 2e^{-\frac{\epsilon^2a^{2r}}{2Na^2}}\leq e^{-\epsilon^2a^{r-1}}.
$$
Taking the union over all subsets of $[n]$ of size $a$, we obtain that $\G$ satisfies~\eqref{eq:middle} with positive probability for all vertex sets of size $a$ as long as
$$
n^ae^{-\epsilon^2a^{r-1}}\leq 1\;,
$$
which is if and only if $a\geq C(\log n)^{1/(r-1)}$ for a constant $C=C(r,\epsilon)$.
\end{proof}


\begin{thebibliography}{99}

\bibitem{APS} N.~Alon, J.~Pach and J.~Solymosi, Ramsey-type theorems with
forbidden subgraphs, Combinatorica 21 (2001), 155--170.

\bibitem{Byellow}
B.~Bollob\'{a}s, \textbf{Modern Graph Theory}, Springer, (1998).

\bibitem{Chud}
M. Chudnovsky, The Erd\H{o}s--Hajnal conjecture -- A survey, J. of Graph Theory 75 (2014), 178--190.

\bibitem{CG} F.~Chung and R.~Graham, Quasi-random hypergraphs,
Proc. Natl.. Acad. Sci. USA 86 (1989), 8175-8177.

\bibitem{CFShyp}
D.~Conlon, J.~Fox and B.~Sudakov, Hypergraph Ramsey numbers, J. Amer. Math. Soc 23 (2010), 247--266.
	
\bibitem{CFS}
D.~Conlon, J.~Fox and B.~Sudakov, Erd\H{o}s--Hajnal-type theorems in hypergraphs, Journal Comb. Theory, Series B 102 (2012), 1142--1154.
	
\bibitem{Erdos}
P. Erd\H{o}s, On extremal problems of graphs and generalized graphs, Israel J. Math. 2 (1964), 183--190.

\bibitem{EH} P.~Erd\H{o}s and A.~Hajnal, Ramsey-type theorems, Discrete Appl. Math. 25 (1989), 37--52.

\bibitem{ER} P.~Erd\H{o}s and R.~Rado, Combinatorial theorems on classifications of subsets of a given set, Proc. London Math. Soc. 3 (1952), 417--439.

\bibitem{GRS} R. L. Graham, B. L. Rothschild and J. H. Spencer, \textbf{Ramsey Theory}, $2^{nd}$ edition, John Wiley and Sons (1980).

\bibitem{KST} T.~K\H{o}v\'{a}ri, V.~S\'{o}s and P.~Tur\'{a}n, On a problem of K.~Zarankiewicz, Colloq. Math. 3 (1954), 50--57.

\bibitem{Nikiforov} V.~Nikiforov, Graphs with many $r$-cliques have large complete $r$-partite subgraphs, Bull. Lond. Math. Soc. 40 (2008), 23--25.

\bibitem{NPRS} B.~Nagle, A.~Poerschke, V.~R\"odl and M.~Schacht, Hypergraph regularity and quasi-randomness, Proceedings of the Twentieth Annual ACM-SIAM Symposium on Discrete Algorithms, SIAM, Philadelphia, PA, (2009), 227-–235.

\bibitem{RRS} C.~Reiher, V.~R\"odl and M.~Schacht,  On a Tur\'{a}n problem in weakly quasirandom 3-uniform hypergraphs, Journal of the European Mathematical Society, to appear.


\bibitem{RodlUniv} V.~R\"{o}dl, On universality of graphs with uniformly distributed edges, Discrete Math. 59 (1986), 125--134.


\bibitem{RS} V.~R\"odl and M. Schacht, Complete partite subgraphs in dense hypergraphs, Random Structures and Algorithms 41 (2012), 557--573.

\bibitem{RW} V.~R\"odl and P.~Winkler, A Ramsey-type theorem for orderings of a graph, SIAM Journal of Discrete Mathematics 2 (1989), 402--406.


\end{thebibliography}
\end{document}